\documentclass[12pt, a4paper]{amsart}
\usepackage{amsmath}
\usepackage{geometry,amsthm,graphics,tabularx,amssymb,shapepar}
\usepackage{amscd}
\usepackage[all,2cell,dvips]{xy}

\newcommand{\BH}{{\mathbb {H}}}

\newcommand{\Ad}{{\mathrm{Ad}}}

\newcommand{\End}{{\mathrm{End}}}

\newcommand{\GL}{{\mathrm{GL}}}

\newcommand{\Hom}{{\mathrm{Hom}}}

\newcommand{\Sp}{{\mathrm{Sp}}}

\newcommand{\tr}{{\mathrm{tr}}}

\newcommand{\vsp}{{\vspace{0.2in}}}

\newcommand{\oO}{\operatorname{O}}

\newcommand{\oU}{\operatorname{U}}

\renewcommand{\sl}{\mathfrak s \mathfrak l}

\newcommand{\C}{\mathbb{C}}
\newcommand{\R}{\mathbb R}

\newcommand{\bH}{\mathbb H}

\newcommand{\ii}{\mathbf{i}}
\newcommand{\jj}{\mathbf{j}}

\newcommand{\la}{\langle}
\newcommand{\ra}{\rangle}

\newcommand{\be}{\begin {equation}}
\newcommand{\ee}{\end {equation}}
\newcommand{\bee}{\begin {equation*}}
\newcommand{\eee}{\end {equation*}}

\theoremstyle{Theorem}

\theoremstyle{Theorem}

\theoremstyle{Theorem}
\newtheorem{lem}{Lemma}[section]

\theoremstyle{Theorem}
\newtheorem{prp}{Proposition}[section]

\newtheorem{thmp}[prp]{Theorem}
\newtheorem{prpp}[prp]{Proposition}

\theoremstyle{Plain}

\theoremstyle{Definition}

\begin{document}

\title[MVW-extensions]{MVW-extensions of real quaternionic classical groups}

\author{Yanan Lin}

\author{Binyong Sun}

\author{Shaobin Tan}

\address{Department of Mathematics, Xiamen University, Xiamen, Fujian, 361005, China}
\email{ynlin@xmu.edu.cn}

\address{Academy of Mathematics and Systems Science, Chinese Academy of
Sciences, Beijing, 100190, China} \email{sun@math.ac.cn}

\address{Department of Mathematics, Xiamen University, Xiamen, Fujian,  361005, China}
\email{tans@xmu.edu.cn}

\subjclass[2000]{20G20}
\keywords{Classical groups, quaternion algebra, MVW-extensions}


\begin{abstract}
Let $G$ be a real quaternionic classical group $\GL_n(\bH)$, $\Sp(p,q)$ or $\oO^*(2n)$. We define an extension $\breve G$ of $G$ with the following property: it contains $G$ as a subgroup of index two, and for every $x\in G$, there is an element $\breve g\in \breve G\setminus G$ such that $\breve g x\breve{g}^{-1}=x^{-1}$. This is similar to Moeglin-Vigneras-Waldspurger's extensions of non-quaternionic classical groups.
\end{abstract}

 \maketitle


\section{Introduction}

We follow the terminology of \cite{SZ}. So by a commutative involutive algebra, we mean a commutative semisimple finite dimensional $\R$-algebra, equipped with a $\R$-algebra involution on it. A commutative
involutive algebra is said to be simple if it is nonzero, and has no involution stable ideal
except for $\{0\}$ and itself. Every commutative involutive algebra is uniquely a product of simple ones, and every simple one is isomorphic to one of the followings:
\begin{equation}\label{simplea}
 (\R, 1_\R),\,\, (\C, 1_\C),\,\, (\C, \bar{\phantom{a}}_\C ),\,\, (\R\times \R, \tau_\R),\,\, (\C\times \C, \tau_\C),
 \end{equation}
where $1_\R$ and $1_\C$ are the identity maps, and $\tau_\R$ and $\tau_\C$ are the maps which interchange the coordinates. In this paper, we use ``$\bar{\phantom{a}}$" (or $\bar{\phantom{a}}_\C$ and $\bar{\phantom{a}}_\bH$) to indicate complex conjugations and quaternion conjugations.

Let $\epsilon=\pm 1$. Let $(A,\tau)$ be a  commutative involutive algebra. Let $E$ be a finitely generated $A$-module. Recall that an $(\epsilon,\tau)$-hermitian form on $E$ is a non-degenerate $\R$-bilinear map
\[
  \la\,,\,\ra_E: E\times E\rightarrow A
\]
satisfying
\begin{equation}\label{herme1}
 \la au, v\ra_E=a\la u,v\ra_E, \quad \la u, v\ra_E=\epsilon \la v,u\ra_E^\tau,\quad a\in A, \,u,v\in E.
\end{equation}
Equipped with such a form, we call $E$ an $(\epsilon,\tau)$-hermitian $A$-module. Then we denote by $\oU(E)$ the group of $A$-module automorphisms of $E$ preserving the form $\la\,,\,\ra_E$. If $(A,\tau)$ is simple, then $E$ is free as an $A$-module, and $\oU(E)$ is a non-quaternionic classical group as in the following table.
\begin{table}[h]
\caption{}
\centering 
\begin{tabular}{c c c c c c c} 
\hline
$(A,\tau)$ & \vline & $(\R, 1_\R)$ & $(\C, 1_\C)$ & $(\C, \bar{\phantom{a}}_\C)$ & $(\R\times \R, \tau_\R)$ & $(\C\times \C, \tau_\C)$\\
\hline 
$\epsilon=1$ &  \vline & $\oO(p,q)$ & $\oO_n(\C)$ & $\oU(p,q)$ & $\GL_n(\R)$ & $\GL_n(\C)$  \\
\hline
$\epsilon=-1$ &  \vline & $\Sp_{2n}(\R)$ & $\Sp_{2n}(\C)$ & $\oU(p,q)$ & $\GL_n(\R)$ & $\GL_n(\C)$  \\
\hline
\end{tabular}
\label{table:nonlin1} 
\end{table}

\vsp

As in \cite{SZ}, we define the MVW-extension $\breve{\oU}(E)$ (named after Moeglin, Vigneras and Waldspurger) of $\oU(E)$ to be the subgroup of
$\GL(E_\R)\times \{\pm 1\}$ consisting of pairs $(g,\delta)$ such
that either
\[
  \delta=1 \quad\textrm{and } \la gu,gv\ra_E=\la u,v\ra_E,\quad  u,v\in  E,
\]
or
\[
  \delta=-1 \quad\textrm{and }\la gu,gv\ra_E=\la v,u\ra_E,\quad  u,v\in E.
\]
Here $E_\R:=E$, viewed as a $\R$-vector space. Every $g\in
\GL(E_\R)$ is automatically $A$-linear if $(g,1)\in
\breve{\oU}(E)$, and is conjugate $A$-linear (with respect to
$\tau$) if $(g,-1)\in \breve{\oU}(E)$.

The MVW-extension $\breve{\oU}(E)$ has the following remarkable property (\cite[Proposition
4.I.2]{MVW87}):  it contains $\oU(E)$ as a subgroup of index two, and for every $x\in \oU(E)$, there is an element $\breve g\in \breve \oU(E)\setminus \oU(E)$ such that $\breve g x\breve{g}^{-1}=x^{-1}$. By character theory and Harish-Chandra's regularity theorem, this has the following interesting consequence in representation theory:  for every element $\breve g\in \breve \oU(E)\setminus \oU(E)$, and for every irreducible Casselman-Wallach representation $\pi$ of $\oU(E)$, its twist $\pi^{\breve g}$ is isomorphic to its contragredient $\pi^\vee$. Here the twist $\pi^{\breve g}$ is the representation $g\mapsto \pi(\breve g g {\breve g}^{-1})$ of $\oU(E)$ on the same space as that of $\pi$. The reader is referred to \cite{Cass} and \cite[Chapter 11]{W2} for details on Casselman-Wallach representations.

\vsp
This paper is aimed to define  similar extensions for real quaternionic classical groups. Denote by $\bH$ a fixed real quaternion algebra. Up to isomorphism, this is the unique central simple division algebra over $\R$ of dimension $4$. Now let $E$ be a finitely generated  $A$-$\bH$-bimodule. Every $A$-$\bH$-bimodule is assumed to satisfy that
\[
  tu=ut, \quad t\in \R, u\in E.
\]
Here the first ``$t$" is viewed as an element of $A$, and the second one is viewed as an element of $\bH$. An $(\epsilon,\tau)$-hermitian form $\la\,,\,\ra_E$ on $E$ is said to be quaternionic if  it further satisfies that \begin{equation}\label{herme2}
  \la u h,v\ra_E=\la u, v \bar{h}\ra_E, \quad  h\in \bH, u,v\in E.
\end{equation}
We equip on $E$ with such a form and then call $E$ a quaternionic $(\epsilon,\tau)$-hermitian $A$-module.

Quaternion $(\epsilon,\tau)$-hermitian $A$-modules are classified as follows.

\begin{prp}\label{classify}
Assume that $(A,\tau)$ is simple. Then every quaternionic $(\epsilon,\tau)$-hermitian $A$-module is isomorphic to exactly one in the following table. Here $p,q,n$ are nonnegative integers, and the quaternionic $(\epsilon,\tau)$-hermitian forms on the spaces in the table are given explicitly in Section \ref{classification}.
\begin{table}[h]
\caption{}
\centering 
\begin{tabular}{c c c c c c c} 
\hline
$(A,\tau)$ & \vline & $(\R, 1_\R)$ & $(\C, 1_\C)$ & $(\C, \bar{\phantom{a}}_\C)$ & $(\R\times \R, \tau_\R)$ & $(\C\times \C, \tau_\C)$\\
\hline 
$\epsilon=1$ &  \vline & $\BH^{p+q}$ & $\C^{2n}\otimes_\C \BH$ & $\C^{p+q}\otimes_\C \BH$ & $\BH^n\oplus\BH^n$ & $\BH^n\oplus \BH^n$  \\
\hline
$\epsilon=-1$ & \vline & $\BH^n$ & $\C^{n}\otimes_\C \BH$ & $\C^{p+q}\otimes_\C \BH$ & $\BH^n\oplus\BH^n$ & $\BH^n\oplus \BH^n$ \\
\hline
\end{tabular}
\label{table:nonlin2} 
\end{table}

\end{prp}

Denote by $\oU(E)$ the group of $A$-$\bH$-bimodule automorphisms of $E$ preserving the form $\la\,,\,\ra_E$. If $(A,\tau)$ is simple and $E$ is as in Table \ref{table:nonlin2}, then $\oU(E)$ is respectively as in the following table.
\begin{table}[h]
\caption{}
\centering 
\begin{tabular}{c c c c c c c} 
\hline
$(A,\tau)$ & \vline & $(\R, 1_\R)$ & $(\C, 1_\C)$ & $(\C, \bar{\phantom{a}}_\C)$ & $(\R\times \R, \tau_\R)$ & $(\C\times \C, \tau_\C)$\\
\hline
$\epsilon=1$ & \vline & $\Sp(p,q)$ & $\Sp_{2n}(\C)$ & $\oU(p,q)$ & $\GL_n(\BH)$ & $\GL_n(\C)$ \\
\hline
$\epsilon=-1$ & \vline & $\oO^*(2n)$ & $\oO_{n}(\C)$ & $\oU(p,q)$ & $\GL_n(\BH)$ & $\GL_n(\C)$ \\
\hline
\end{tabular}
\label{table:nonlin3} 
\end{table}

In general, we define the MVW-extension of $\oU(E)$ to be
\[
  \breve{\oU}(E):=\breve{\oU}(E_A)\cap (\GL(E_\bH)\times \{\pm 1\}),
\]
where $E_A$ is the underlying $(\epsilon,\tau)$-hermitian $A$-module of $E$ (forgetting the $\bH$-module structure), and $E_\bH:=E$, viewed as a right $\bH$-vector space.

The main result of this paper is
\begin{thmp}\label{geo} Let $E$ be a quaternionic $(\epsilon,\tau)$-hermitian $A$-module. Then $\breve{\oU}(E)$ contains $\oU(E)$ as a subgroup of index two, and for every $x\in \oU(E)$, there is an element $\breve
g\in \breve{\oU}(E)\setminus \oU(E)$ such that
$\breve{g}x\breve{g}^{-1}=x^{-1}$.
\end{thmp}

\noindent {\bf Remarks}: (a) MVW-extensions are defined for p-adic non-quaternionic classical groups as well (cf. \cite{Sun2}). One may also define MVW-extensions for p-adic quaternionic general linear groups (cf. \cite[Lemma 3.1]{Ra}). But by a private communication whit D. Prasad, it seems that it is not possible to define MVW-extensions with desired properties for p-adic quaternionic unitary groups.

(b) MVW-extensions have many applications to representation theory of classical groups. For examples, they are used to prove Multiplicity One Theorems for non-quaternionic classical groups (\cite{AGRS,SZ}), and multiplicity preservation for local theta correspondences (\cite{LST}). They are also used to prove that local theta correspondence maps hermitian representations to hermitian representations (cf. \cite{Prz, Sun1}).

(c) The Lie algebra analog of the second statement of Theorem \ref{geo} also holds. Namely, for every $x$ in the Lie algebra of $\oU(E)$, there is an element $\breve g\in \breve{\oU}(E)\setminus \oU(E)$ such that
$\Ad_{\breve{g}} x=-x$.

(d) If $\epsilon=1$ and $(A,\tau)=(\R, 1_\R)$ or $(\C, 1_\C)$, then $ \breve{\oU}(E)=\oU(E)\times\{\pm 1\}$. Consequently, every irreducible Casselman-Wallach representation of $\Sp(p,q)$ or $\Sp_{2n}(\C)$ is self-cotragredient.

\vsp \noindent Acknowledgements: The authors thank C.-F. Nien for pointing out the reference \cite{Ra}. The work was supported by NSFC Grant 10931006.

\section{The classification of quaternionic $(\epsilon,\tau)$-hermitian $A$-modules}\label{classification}

We fix a $\R$-algebra embedding $\C\hookrightarrow \bH$ and thus view $\C$ as a subalgebra of $\bH$. Note that such an embedding is unique up to conjugations by $\bH^\times$, and that complex conjugation is consistent with the quaternion conjugation under the embedding. Write $\ii=\sqrt{-1}\in \C$, and fix an element $\jj\in \bH$ such that $\jj^2=-1$ and $\jj x \jj^{-1}=\bar x$ for all $x\in \C$. Denote by $\tr_{\bH/\R}:\bH\rightarrow \R$ the reduced trace of $\bH$.

Let $(A,\tau)$ be a commutative involutive algebra as in the introduction. In general, if $(E_i,\la\,,\,\ra_{E_i})$ is an $(\epsilon_i,\tau)$-hermitian $A)$-module ($\epsilon_i=\pm 1$, $i=1,2$), for every $\varphi\in \Hom_A(E_1,E_2)$, we define $\varphi^\tau\in \Hom_A(E_2,E_1)$ by requiring that
\begin{equation}\label{deftau}
   \la \varphi u,v\ra_{E_2}= \la u,\varphi^\tau v\ra_{E_1}, \quad u\in E_1, \,v\in E_2.
\end{equation}
Note that $(\varphi^\tau)^\tau=\epsilon_1 \epsilon_2 \varphi$.

Let $(E,\la\,,\,\ra_E)$ be a quaternionic $(\epsilon,\tau)$-hermitian $A$-module as in the introduction.

\subsection{The case of $(A,\tau)=(\R,1_\R)$}\label{c0}
Assume that $(A,\tau)=(\R,1_\R)$. Define a $\R$-bilinear map
\[
  \la\,,\,\ra_{E,\bH}: E\times E\rightarrow \bH
\]
by requiring that
\[
  \tr_{\bH/\R}(\la u,v \ra_{E,\bH}h)=\la u,vh \ra_E, \quad h\in \bH,\,u,v\in E.
\]
Then $ \la\,,\,\ra_{E,\bH}$ is a (non-degenerate) $(\epsilon,\bar{\,\,}_\bH)$-hermitian form on $E$, namely, it satisfies that
\[
  \la u, vh\ra_{E,\bH}=\la u,v\ra_{E,\bH}h, \quad \la u, v\ra_{E,\bH}=\epsilon \overline{\la v,u\ra}_{E,\bH},\quad h\in \bH, \,u,v\in E.
\]
On the other hand, $\la\,,\,\ra_{E,\bH}$ determines $\la\,,\,\ra_E$ by the formula
\[
  \la u,v \ra_E=\tr_{\bH/\R}(\la u,v \ra_{E,\bH}), \,u,v\in \bH.
\]
Therefore, there is a one-one correspondence between the set of quaternionic $(\epsilon,\tau)$-hermitian forms, and the set of $(\epsilon,\bar{\,\,}_\bH)$-hermitian forms, on $E$. From the well known classification of $(\epsilon,\bar{\,\,}_\bH)$-hermitian forms, we conclude that
\begin{prp}\label{cl1}
Assume that $(A,\tau)=(\R,1_\R)$. If $\epsilon=1$, then every quaternionic $(\epsilon,\tau)$-hermitian $A$-module is isomorphic to a unique $(\bH^{p+q}, \la \,,\,\ra_{\bH^{p,q}})$, where $p,q\geq 0$, and
\[
   \la (u_1,u_2,\cdots,u_{p+q}),\, (v_1, v_2,\cdots, v_{p+q})\ra_{\bH^{p,q}}=\tr_{\bH/\R}(\sum_{i\leq p} \bar{u_i}v_i-\sum_{j>p}\bar{u_j}v_j).
\]
If $\epsilon=-1$, then every quaternionic $(\epsilon,\tau)$-hermitian $A$-module is isomorphic to a unique $(\bH^n, \la \,,\,\ra_{\bH^n,\jj})$, where $n\geq 0$, and
\[
   \la (u_1,u_2,\cdot,u_{n}),\, (v_1, v_2,\cdots, v_n)\ra_{\bH^n,\jj}=\tr_{\bH/\R}(\sum_{i=1}^n \bar{u_i}\jj v_i).
\]
\end{prp}

\subsection{The case of $(A,\tau)=(\C, 1_\C)$}\label{c1}

Assume that $(A,\tau)=(\C,1_\C)$. View $\bH$ as a $\C$-$\bH$-bimodule. It is a quaternionic $(-1,1_\C)$-hermitian $\C$-module under the form
\begin{equation}\label{csform}
   \begin{array}{rcl}
     \la\,,\,\ra_{\bH,-1}: \bH\times \bH& \rightarrow & \C,\\
      (a+b \jj, c+d\jj)& \mapsto & ad-bc, \quad a,b,c,d\in \C.
  \end{array}
\end{equation}

Put
\[
  F:=\Hom_{\C,\bH}(\bH, E),
\]
which is clearly a $\C$-vector space. It is a $(-\epsilon,1_\C)$-hermitian $\C$-module under the form
\[
  \la \varphi,\psi\ra_F:=\psi^\tau\circ\varphi\in \End_{\C,\bH}(\bH)=\C.
\]
On the other hand, $(E,\la\,,\,\ra_E)$ is determined by the $(-\epsilon,1_\C)$-hermitian $\C$-module $(F,\la\,,\,\ra_F)$ since
\[
  E=F\otimes_\C \bH
\]
as a $\C$-$\bH$-bimodule, and
\[
  \la \,,\,\ra_E=\la \,,\,\ra_F\,\la\, ,\, \ra_{\bH,-1}.
\]
Therefore, there is a one-one correspondence between the set of isomorphism classes of quaternionic $(\epsilon,1_\C)$-hermitian $\C$-module, and the set of isomorphism classes of $(-\epsilon,1_\C)$-hermitian $\C$-module. By the classification of $(-\epsilon,1_\C)$-hermitian $\C$-modules (that is, complex symplectic spaces or complex orthogonal spaces), we have
\begin{prpp}\label{cl2}
Assume that $(A,\tau)=(\C,1)$. If $\epsilon=1$, then every quaternionic $(\epsilon,\tau)$-hermitian $A$-module is isomorphic to a unique
\[
  (\C^{2n}\otimes_\C \bH, \,\la \,,\,\ra_{\C^{2n},-1}\otimes \la \,,\,\ra_{\bH,-1}),
\]
where $n\geq 0$,  and
\[
         \la (u_1,u_2,\cdots,u_{2n}),\, (v_1, v_2,\cdots, v_{2n})\ra_{\C^{2n},-1}=\sum_{i=1}^n (u_i v_{n+i}-u_{n+i}v_i).
\]
If $\epsilon=-1$, then every quaternionic $(\epsilon,\tau)$-hermitian $A$-module is isomorphic to a unique
\[
  (\C^{n}\otimes_\C \bH, \,\la \,,\,\ra_{\C^{n},1}\otimes \la \,,\,\ra_{\bH,-1}),
\]
where $n\geq 0$,  and
\[
         \la (u_1,u_2,\cdots,u_{n}),\, (v_1, v_2,\cdots, v_{n})\ra_{\C^{n},1}=\sum_{i=1}^n u_i v_i.
\]
\end{prpp}

\subsection{The case of $(A,\tau)=(\C,\bar{\,\,}_\C)$}\label{c2}
Assume that $(A,\tau)=(\C,\bar{\,\,}_\C)$. The argument is similar to that of the last section. View $\bH$ as a $\C$-$\bH$-bimodule as before. It is a quaternionic $(1,\bar{\,\,}_\C)$-hermitian $\C$-module under the form
\[
   \begin{array}{rcl}
     \la\,,\,\ra_{\bH,1}: \bH\times \bH& \rightarrow & \C,\\
      (a+b \jj, c+d\jj)& \mapsto & a\bar c+b\bar d, \quad a,b,c,d\in \C.
  \end{array}
\]

The space $F:=\Hom_{\C,\bH}(\bH, E)$ is an $(\epsilon,\bar{\,\,}_\C)$-hermitian $\C$-module under the form
\[
  \la \varphi,\psi\ra_F:=\psi^\tau\circ\varphi\in \End_{\C,\bH}(\bH)=\C.
\]
On the other hand, $(E,\la\,,\,\ra_E)$ is determined by $(F,\la\,,\,\ra_F)$ since
\[
  E=F\otimes_\C \bH
\]
as a $\C$-$\bH$-bimodule, and
\[
  \la \,,\,\ra_E=\la \,,\,\ra_F\,\la\, ,\, \ra_{\bH,1}.
\]
Therefore, by the classification of $(\epsilon,\bar{\,\,}_\C)$-hermitian $\C$-modules (that is, the usual hermitian or skew hermitian spaces), we have
\begin{prpp}\label{cl3}
Assume that $(A,\tau)=(\C,\bar{\,\,}_\C)$. If $\epsilon=1$, then every quaternionic $(\epsilon,\tau)$-hermitian $A$-module is isomorphic to a unique
\[
  (\C^{p+q}\otimes_\C \bH,\, \la \,,\,\ra_{\C^{p,q},1}\otimes \la \,,\,\ra_{\bH,1}),
\]
where $p,q\geq 0$, and
\[
         \la (u_1,u_2,\cdots,u_{p+q}),\, (v_1, v_2,\cdots, v_{p+q})\ra_{\C^{p,q},1}=\sum_{i\leq p} u_i\bar{v_i}-\sum_{j>p}u_j \bar{v_j}.
\]
If $\epsilon=-1$, then every quaternionic $(\epsilon,\tau)$-hermitian $A$-module is isomorphic to a unique
\[
  (\C^{p+q}\otimes_\C \bH,\, \la \,,\,\ra_{\C^{p,q},-1}\otimes \la \,,\,\ra_{\bH,1}),
\]
where $p,q\geq 0$, and
\[
         \la (u_1,u_2,\cdots,u_{p+q}),\, (v_1, v_2,\cdots, v_{p+q})\ra_{\C^{p,q},-1}=\sum_{i\leq p} \ii\, u_i\bar{v_i}-\sum_{j>p}\ii\, u_j \bar{v_j}.
\]
\end{prpp}

\subsection{The case of $(A,\tau)=(\R\times \R, \tau_\R)$}\label{c3}

Assume that $(A,\tau)=(\R\times \R, \tau_\R)$. Write $e_1=(1,0)$ and $e_2=(0,1)$. Then we have a decomposition
\[
  E=E_1\oplus E_2
\]
of $E$ as a right $\bH$-vector space, where $E_1=e_1 E$ and $E_2=e_2 E$. It is easy to see that there is a unique non-degenerate $\R$-bilinear map
\[
   \la \,,\,\ra_{E_1,E_2}: E_1\times E_2\rightarrow \R
\]
satisfying
\[
   \la uh,v\ra_{E_1,E_2}=\la u,v\bar h\ra_{E_1,E_2},\quad u\in E_1, \,v\in E_2, \, h\in \bH
\]
such that
\[
  \la u_1+u_2, v_1+v_2  \ra_E=( \la u_1,v_2\ra_{E_1,E_2},\, \epsilon \la v_1, u_2\ra_{E_1,E_2}), \quad u_1,v_1
  \in E_1, \, u_2,v_2\in E_2.
\]
Note that the right $\bH$-vector space $E_2$ is determined by  the right $\bH$-vector space $E_1$ together with the paring $\la \,,\,\ra_{E_1,E_2}$. Namely, $E_2=\Hom_\R(E_1,\R)$ as right $\bH$-vector spaces. Therefore, we have
\begin{prpp}\label{cl4}
Assume that $(A,\tau)=(\R\times \R, \tau_\R)$. Then every quaternionic $(\epsilon,\tau)$-hermitian $A$-module is isomorphic to a unique
\[
  (\bH^n\oplus \bH^n,\, \la \,,\,\ra_{\bH^{2n},\R, \epsilon}),
\]
where $n\geq 0$, and
\[
         \la (u_1,u_2,\cdots,u_{2n}),\, (v_1, v_2,\cdots, v_{2n})\ra_{\bH^{2n},\R, \epsilon}=(\tr_{\bH/\R}(\sum_{i=1}^n \bar{u_i}v_{n+i}),\,\epsilon \tr_{\bH/\R}(\sum_{i=1}^n \bar{v_i}u_{n+i})).
\]
\end{prpp}

\subsection{The case of $(A,\tau)=(\C\times \C, \tau_\C)$}\label{c4}

Note that every finitely generated $\C$-$\bH$-bimodule is isomorphic to some $\bH^n$. Now the argument is the same as in last section. We skip the details and record the result:
\begin{prpp}\label{cl5}
Assume that $(A,\tau)=(\C\times \C, \tau_\C)$. Then every quaternionic $(\epsilon,\tau)$-hermitian $A$-module is isomorphic to a unique
\[
  (\bH^n\oplus \bH^n,\, \la \,,\,\ra_{\bH^{2n},\C, \epsilon}),
\]
where $n\geq 0$, and
\[
         \la (u_1,u_2,\cdots,u_{2n}),\, (v_1, v_2,\cdots, v_{2n})\ra_{\bH^{2n},\C, \epsilon}=(\sum_{i=1}^n \la u_i, v_{n+i}\ra_{\bH,-1},\,\epsilon \sum_{i=1}^n \la v_i, u_{n+i}\ra_{\bH,-1}).
\]
Here the complex symplectic form $\la\, ,\,\ra_{\bH,-1}:\bH\times \bH\rightarrow \C$ is the one in \eqref{csform}.

\end{prpp}

We get Proposition \ref{classify} by combining Propositions \ref{cl1}-\ref{cl5}.

\section{Quaternionic $(\epsilon,\tau)$-hermitian $(\sl_2, A)$-modules}

By an $(\epsilon,\tau)$-hermitian $(\sl_2,A)$-module, we mean an $(\epsilon,\tau)$-hermitian $A$-module $E$, together with a Lie algebra action
\begin{equation}\label{actionsl}
  \sl_2(\R)\times E\rightarrow E, \quad x,v\mapsto xv,
\end{equation}
which is $\R$-linear on the first factor,  $A$-linear on the second
factor, and satisfies
\[
  \la xu,v\ra_E+\la u,xv\ra_E=0, \quad x\in \sl_2(\R),\,u,v\in E.
\]
If  $E$ is furthermore a quaternionic $(\epsilon,\tau)$-hermitian $A$-module, and the action \eqref{actionsl} is $\bH$-linear on the second factor as well, we say that $E$ is a quaternionic $(\epsilon,\tau)$-hermitian $(\sl_2,A)$-module.

For every positive integer $d$, fix an irreducible real $\sl_2(\R)$-module $F_d$ of dimension $d$, which is unique up to isomorphism. We also fix a non-zero $\sl_2(\R)$-invariant bilinear map
\[
  \la \,,\,\ra_{F_d}: F_d\times F_d\rightarrow \R,
\]
which is unique up to scalar multiplication. The form $\la \,,\,\ra_{F_d}$ is non-degenerate. It is symmetric if $d$ odd, and is skew symmetric if $d$ is even. The space
\[
  A\otimes_\R F_d
\]
is obviously an $((-1)^{d-1},\tau)$-hermitian $(\sl_2,A)$-module with the form
\[
  \la a\otimes u, b\otimes v\ra_{A\otimes_\R F_d}:=ab^\tau \la u,  v\ra_{F_d}
\]

Now let $E$ be a quaternionic $(\epsilon,\tau)$-hermitian $(\sl_2,A)$-module. Denote by $E(d)$ the sum of irreducible $\sl_2(\R)$-submodules of $E$ which are isomorphic to $F_d$. It is again a quaternionic $(\epsilon,\tau)$-hermitian $(\sl_2,A)$-module, and we have an orthogonal decomposition
\[
  E=\bigoplus_{d\geq 1} E(d).
\]
The space
\[
 E(d)^\circ:=\Hom_{\sl_2(\R)}(F_d, E(d))=\Hom_{\sl_2(\R), A}(A\otimes_\R F_d, E(d))
\]
is checked to be  a quaternionic $((-1)^{d-1}\epsilon,\tau)$-hermitian $A$-module under the form
\[
  \la \varphi, \psi\ra_{E(d)^\circ}:=\psi^\tau\circ\varphi\in \End_{\sl_2(\R),A}(A\otimes_\R F_d)=A.
\]

On the other hand, the quaternionic $(\epsilon,\tau)$-hermitian $(\sl_2,A)$-module $E(d)$ is determined by the quaternionic $((-1)^{d-1}\epsilon,\tau)$-hermitian $A$-module $E(d)^\circ$ since
\[
  E(d)=E(d)^\circ\otimes_\R F_d
\]
as a $\sl_2(\R)$-$A$-$\bH$-module, and
\[
   \la \,,\,\ra_E|_{E(d)\times E(d)}=\la\, , \,\ra_{E(d)^\circ}\otimes \la \,,\,\ra_{F_d}.
\]

In conclusion, we have
\begin{equation}\label{decome}
  E=\bigoplus_{d\geq 1} E(d)^\circ\otimes_\R F_d
\end{equation}
as a quaternionic $(\epsilon,\tau)$-hermitian $(\sl_2,A)$-module.

\section{Proof of Theorem \ref{geo}}

Let $E$ be an  $(\epsilon,\tau)$-hermitian $A$-module, or a quaternionic $(\epsilon,\tau)$-hermitian $A$-module, or an $(\epsilon,\tau)$-hermitian $(\sl_2,A)$-module, or a quaternionic $(\epsilon,\tau)$-hermitian $(\sl_2,A)$-module. We define $E_\tau$, which is a space with the same kinds of structures as that of $E$, as follows. In any case, $E_\tau=E$ as a real vector space, and when $E$ is quaternionic, $E_\tau=E$ as a right $\bH$-vector space. For every $v\in E$,
write $v_\tau:=v$, viewed as a vector in $E_\tau$. Then the scalar multiplication on $E_\tau$ is
\[
  a v_\tau:=(a^\tau v)_\tau, \quad a\in A, v\in E,
\]
and the hermitian form is
\[
  \la u_\tau, v_\tau\ra_{E_\tau}:=\la v,u\ra_E, \quad u,v \in E.
\]
When $E$ is equipped with a $\sl_2(\R)$-action, the $\sl_2(\R)$-action on $E_\tau$ is given by
\[
  \mathbf h v_\tau:=(\mathbf h v)_\tau, \quad \mathbf e
  v_\tau:=-(\mathbf e v)_\tau, \quad \mathbf f v_\tau:=-(\mathbf f
  v)_\tau, \quad v\in E,
\]
where
\[
  \mathbf h:= \left[
                   \begin{array}{cc} 1&0\\ 0&-1\\
                   \end{array}
  \right], \quad
  \mathbf e:= \left[
                   \begin{array}{cc} 0&1\\ 0&0\\
                   \end{array}
  \right],\quad
  \mathbf f:=\left[
                   \begin{array}{cc} 0&0\\ 1&0\\
                   \end{array}
  \right],
\]
which form a basis of the Lie algebra $\sl_2(\R)$.

\begin{lem}
If is $E$ be an  $(\epsilon,\tau)$-hermitian $A$-module, or a quaternionic $(\epsilon,\tau)$-hermitian $A$-module, or an $(\epsilon,\tau)$-hermitian $(\sl_2,A)$-module, or a quaternionic $(\epsilon,\tau)$-hermitian $(\sl_2,A)$-module. Then, $E_\tau$ is isomorphic to $E$.
\end{lem}
\begin{proof}
The lemma is proved in \cite[Proposition 1.2 and 1.5]{Sun2} when  $E$ is an $(\epsilon,\tau)$-hermitian $A$-module, or a $(\epsilon,\tau)$-hermitian $(\sl_2,A)$-module.

Without lose of generality, assume that $(A,\tau)$ is simple. The lemma follows easily from Proposition \ref{classify} if $E$ is a  quaternionic $(\epsilon,\tau)$-hermitian $A$-module.

Now assume that $E$ is a  quaternionic $(\epsilon,\tau)$-hermitian $(\sl_2,A)$-module. Write
\[
   E=\bigoplus_{d\geq 1} E(d)^\circ\otimes_\R F_d
\]
as in \eqref{decome}. Recall that $E(d)^\circ$ is a quaternionic $((-1)^{d-1}\epsilon,\tau)$-hermitian $A$-module, and $F_d$ is a $((-1)^{d-1}, 1_\R)$-hermitian $(\sl_2,\R)$-module. Therefore
\begin{eqnarray*}
  E_\tau &\cong& (\bigoplus_{d\geq 1} E(d)^\circ\otimes_\R F_d)_\tau\\
   &\cong& \bigoplus_{d\geq 1} (E(d)^\circ)_\tau\otimes_\R (F_d)_\tau \\
   &\cong& \bigoplus_{d\geq 1} E(d)^\circ\otimes_\R F_d   \\
   &=& E.
\end{eqnarray*}

\end{proof}

Now we are ready to prove Theorem \ref{geo}. Let $E$ be a quaternionic $(\epsilon,\tau)$-hermitian $A$-module. Note that the isomorphism $E\cong E_\tau$ amounts to saying that there is an element of the form $(g,-1)$ in $\breve{\oU}(E)$, that is,  $\oU(E)$ has index two in $\breve{\oU}(E)$. This proves the first statement of Theorem \ref{geo}.

Let $x\in \oU(E)$. Write $x=su$ for the Jordan decomposition of $x$, where $s\in \oU(E)$ is semisimple, and $u\in \oU(E)$ is unipotent. Denote by $A_s$ the subalgebra of $\End_{A,\bH}(E)$ generated by $s$, $s^{-1}$ and scalar multiplications by $A$. Recall that as a special case of \eqref{deftau}, we have defined an involution $\tau$ on $\End_{A}(E)$. The algebra $A_s$ is $\tau$-stable, and the pair $(A_s,\tau)$ is a commutative involutive algebra. Write $E_s:=E$, viewed as an $A_s$-$\bH$-bimodule, equipped with the $\R$-bilinear map
\[
   \la\,,\,\ra_{E_s}: E_s \times E_s \rightarrow A_s
\]
such that
\[
    \tr_{A_s/\R}(a\la u, v\ra_{E_s})= \tr_{A/\R}(\la a u, v\ra_{E}), \quad u,v\in E, \,\,a\in A_s.
\]
Then $E_s$ becomes a quaternion $(\epsilon,\tau)$-hermitian $A_s$-modules, and we have that
\[
  \breve{\oU}(E_s)=\{(g,\delta)\in \breve{\oU}(E)\mid gsg^{-1}=s^\delta\}.
\]

Use Jacobson-Morozov Theorem, we choose an action of $\sl_2(\R)$ on $E_s$ such that it makes $E_s$ a quaternionic
$(\epsilon,\tau)$-hermitian $(\sl_2,A_s)$-module, and that the exponential of
the action of $\mathbf e$ coincides with $u$. Since $(E_s)_\tau$ is isomorphic to $E_s$ as a quaternionic
$(\epsilon,\tau)$-hermitian $(\sl_2,A_s)$-module, there is an element $\breve g=(g,-1)\in \breve{\oU}(E_s)$ such that $gug^{-1}=u^{-1}$. This proves  Theorem \ref{geo} since $\breve g\in \breve{\oU}(E)\setminus \oU(E)$ and  $\breve{g}x\breve{g}^{-1}=x^{-1}$.

\end{document}